\documentclass[12pt,reqno,a4paper]{amsart}
\usepackage[latin1]{inputenc}
\usepackage{amsmath}
\usepackage{amsthm}
\usepackage{amsfonts}
\usepackage{amssymb}
\usepackage[left=2cm,right=2cm,top=2cm,bottom=2cm]{geometry}
\usepackage{xcolor}
\usepackage{enumitem}

\makeatletter

\newcommand*{\eqdef}{\mathrel{\vcenter{\baselineskip0.5ex \lineskiplimit0pt
                     \hbox{\scriptsize.}\hbox{\scriptsize.}}}%
                     =}
\makeatother
\def\R{\mathbb{R}}

\def\N{{\mathbb N}}
\def\Z{{\mathbb Z}}
\def\R{{\mathbb R}}

\newtheorem{lemma}{Lemma}
\newtheorem{theorem}{Theorem}
\newtheorem{proposition}{Proposition}
\newtheorem{definition}{Definition}
\newtheorem{corollary}{Corollary}
\newtheorem{remark}{Remark}

\title[{Multiplicity results for semilinear elliptic problems with Neumann conditions}]{Multiplicity results and qualitative properties for semilinear elliptic problems with Neumann boundary conditions}

\author[O. Agudelo]{Oscar Agudelo}
\address{\noindent O. Agudelo - NTIS, Department of Mathematics,
Zapadoceska Univerzita v Plzni, Plzen, Czech Republic.}
\email{oiagudel@ntis.zcu.cz}

\author[S. Correa]{Santiago Correa}
\address{\noindent S. Correa - Escuela de Matem\'{a}ticas, Universidad Nacional de Colombia, Sede Medell\'{i}n Medell\'{i}n, Colombia.}
\email{scorreac@unal.edu.co}

\author[D. Restrepo]{Daniel Restrepo}
\address{\noindent D. Restrepo - Escuela de Matem\'{a}ticas, Universidad Nacional de Colombia, Sede Medell\'{i}n Medell\'{i}n, Colombia.}
\email{daerestrepomo@unal.edu.co}

\author[C. V\'{e}lez]{Carlos V\'{e}lez}
\address{\noindent C. V\'{e}lez - Escuela de Matem\'{a}ticas, Universidad Nacional de Colombia, Sede Medell\'{i}n Medell\'{i}n, Colombia.}
\email{cauvelez@unal.edu.co}

\begin{document}
	
\maketitle
	
\begin{abstract}
In this paper we study multiplicity and qualitative behavior of solutions for semilinear elliptic problems with neumann boundary condition and asymptotically linear smooth nonlinearity. We provide sufficient conditions on the number of eigenvalues the derivative of the nonlinearity crosses to guarantee existence of at least five nontrivial solutions. The techniques we use are a combination of minimization, Leray-Schauder degree, Morse Theory and Reduction method a la Castro-Lazer.
\end{abstract}

\section{Introduction}

In this paper we study existence of multiple nontrivial solutions for the boundary value problem (BVP)
\begin{equation}
\label{NEUMANN}
\left\{
\begin{aligned}
-\Delta u =& f(u) &\hbox{in} \quad \,\,\,\,\,\Omega,\\
\frac{\partial u}{\partial \nu} = &0 &\hbox{on} \quad \partial \Omega,
\end{aligned}
\right.
\end{equation}
where $\Omega \subset \R^N$ ($N\geq 1$) is a smooth bounded domain with outer unit normal vector $\nu :\partial \Omega \to S^{N-1}$. The operator $\Delta:= \sum_{i=1}^N \partial^2_{x_i}$ is the Laplace operator and $f:\R\to \R$ is a continuously diffe\-rentiable function, asymptotically linear at $\pm \infty$ and satisfying some further assumptions, to be specified later.

\medskip
In order to describe the set of hypotheses on $f$, set
$$
f'(+\infty):= \lim \limits_{t\to + \infty} \frac{f(t)}{t} \quad \hbox{and} \quad f'(-\infty):= \lim \limits_{t\to - \infty} \frac{f(t)}{t},
$$
so that $f$ is asymptotically linear if and only if $f'(\pm \infty) \in  \R$.

\medskip
Next, let
$$
0=\mu_0 < \mu_1 \leq \mu_2 \leq \cdots \leq \mu_k \leq \mu_{k+1}\leq \cdots
$$
denote the sequence of eigenvalues of the linear BVP
\begin{equation}
\label{EIGENVALUENEUMANN}
\left\{
\begin{aligned}
-\Delta \varphi =& \mu \varphi  &\hbox{in} \quad \,\,\,\,\,\Omega,\\
\frac{\partial \varphi}{\partial \nu} = &0 &\hbox{on} \quad \partial \Omega .
\end{aligned}
\right.
\end{equation}

\medskip
The multiplicity results presented here deal with the case when $f'(\pm \infty)$ are finite and cross the same number of eigenvalues $\mu_k$. 

\medskip
Problem \eqref{NEUMANN} in the Dirichlet setting, has been extensively studied. We refer the reader for instance to \cite{BERESTICKYDEFIGUEIREDO81,
	CASTROLAZER79,COSTASILVA95,DEFIGUEIREDO87,
	HIRANONISHIMURA93,LANDESMANROBINSON95,LIANGSU09,LIUSHIBO09}
and references there in.

\medskip
As for the existence of nontrivial solutions of \eqref{NEUMANN}, we refer the reader to \cite{MAWHINWARDWILLEM86}, where the nonlinearity is assumed to be nondecreasing and $f'(\pm \infty) \leq \mu_1$ and to \cite{MAWHIN87} for a generali\-zation of the latter work.

\medskip 
In \cite{IANNACCINKASHAMA89,KUO96} existence of nontrivial solutions of \eqref{NEUMANN} in the resonant case is also treated using Landesman-Lazer type conditons.

\medskip
As for multiplicity of nontrivial solutions of \eqref{NEUMANN}, we refer the reader to \cite{LI03,LILI04,QIAN05}, where infinitely many solutions are obtained in the superlinear case and under either some oscillatory or symmetry assumptions on the nonlinearity. 

\medskip

In the resonant case, the Poincare inequality is not available and this makes the analysis of \eqref{NEUMANN} more involved. In \cite{FILIPPAKISPAPAGEORGIOU10}
the authors consider resonance respect to $\mu_0=0$ and under a sign condition on the nonlinearity existence of three nontrivial solutions is proved. In \cite{TANGWU03} the authors assume that the potential of the nonlinearity is anticoercive and prove the existence of two nontrivial solutions.

\medskip
In \cite{GASINSKIPAPAGEORGIOU12} the authors consider the case where resonance at zero and at infinity occur, but respect to different eigenvalues. In this work, a combination of critical point theory, Lyapunov-Schmidt reduction and Morse Theory methods is used to establish the existence of at least five nontrivial solutions of \eqref{NEUMANN}. The authors in \cite{GASINSKIPAPAGEORGIOU14} obtained two nontrivial solutions of \eqref{NEUMANN}, under similar hypotheses as in \cite{GASINSKIPAPAGEORGIOU12}, but allowing for resonance at zero and at infinity respect to the same eigenvalue.

\medskip

In this work we prove the existence of at least five nontrivial solutions to \eqref{NEUMANN} under different hypotheses to those assumed in \cite{GASINSKIPAPAGEORGIOU12} and \cite{GASINSKIPAPAGEORGIOU14} (see Sections 4 and 5 below for precise statements). We also provide, for a given positive integer $k$, conditions under which problem \eqref{NEUMANN} has at least $k$ nontrivial solutions (see Section 3 below for precise statement).

\bigskip
%
%

We prove our results using a combination of variational techniques such as a reduction procedure alla Castro-Lazer, see \cite{CASTROLAZER79}, Morse Theory and computation of critical groups see \cite{CHANGKC93} and truncation methods. 

\medskip


\medskip

The paper is organized as follows. Section 2 contains some preliminary background intended to make the presentation as self-contained as possible. Section 3  is devoted to some technical lemmas about existence of solutions which are either local minima or of mountain pass type. Section 4 contains the degree computations of the solutions found in Section 3 and finally Section 5 contains the reduction procedure and the proof of our main result. 

\section{Background}

\section{Lemmas}

\begin{proposition}[Qualitative behaviour]\label{qualitative}
	Let $f:\R\times \overline{\Omega} \to\R$ be a continuous function, with $\frac{\partial f}{\partial t}$ continuous, let $\Omega$ be an open bounded subset of $\R^N$ with $C^1$-boundary and let $\hat{n}$ be the unit outward normal vector to $\Omega$. If $u$ is classical solution of the Neumann boundary problem
	\begin{equation}\label{eq Neumann problem}
	\begin{cases}
	-\Delta u=f(u) \quad \text{in} \; \Omega \, ,\\
	\frac{\partial u}{\partial \hat{n}} = 0    \qquad \hspace{2mm} \enspace \text{on}\, \, \: \partial \Omega,
	\end{cases}
	\end{equation}
	and $\alpha \in \R$ is such that $f(\alpha,x)\leq 0$ for each $x\in \overline{\Omega}$, then $\max_{\overline{\Omega}} u =\alpha$ only if $u$ is constant.
\end{proposition}
\begin{proof}
	Let us suppose by contradiction that $u$ is nonconstant and that $\max_{\overline{\Omega}} u =\alpha$. By the boundedness of $\frac{\partial f}{\partial t}(s,x)$ in $s$ in the range of $u$ and $x\in \overline{\Omega}$  we can find $m>0$ such that the function $t\to mt+f(t,x)$ is increasing in that interval for each $x\in \overline{\Omega}$. Hence, it follows that 
	\begin{equation*}
	(\Delta -m)(u-\alpha)=m\alpha-(mu+f(u,x))\geq f(\alpha,x)+m\alpha-(mu+f(u,x))\geq 0
	\end{equation*}
	in $\Omega$. If we fix $x_0$ as a global maximizer of $u-\alpha$ on $\partial \Omega$, then $u(x_0)-\alpha= 0$. On the other hand, Theorem 3.2 of Gilbarg-Trudinger implies that $u-\alpha$ cannot achieve its maximum in the interior of $\Omega$, hence, Lemma 3.4 of the same reference implies that $\frac{\partial u}{\partial \hat{n}}(x_0)>0$, which contradicts the fact that $u$ is a solution of our boundary value problem.
\end{proof}
\begin{remark}
	We have a similar result if we suppose that $\min_{\overline{\Omega}} u=\alpha$ with $f(\alpha,x)\geq 0$.
\end{remark}
As a straightforward conclusion of the proposition (in the autonomous case) we have that if $u$ is a non-constant solution of our problem then $f(\max_{\overline{\Omega}} u)>0$ and that $f(\min_{\overline{\Omega}} u)<0$.
\begin{remark}\label{eigenvalues}
	Let us recall some important properties of the spectrum of the Laplacian with Neumann boundary condition (see Santiago Correa thesis).
	\begin{itemize}
		\item $\lambda_1=0$ is the first (principal) eigenvalue of this operator and it is simple.
		\item The eigenspace associated with $\lambda_1$ is formed exclusively by constant functions.
		\item If we denote $E_i$ the eigenspace associated with the eigenvalue $\lambda_i$ we have that for $k\geq 1$ and every $v\in Y:=\bigoplus_{i=1}^k E_i$ it holds
		\begin{equation*}
		\Vert \nabla v\Vert_{L^2(\Omega)}^2\leq \lambda_k \Vert v\Vert_{L^2(\Omega)}^2.
		\end{equation*}
		On the other hand, for every $v\in Y^\bot$ we have that
		\begin{equation*}
		\Vert \nabla v\Vert_{L^2(\Omega)}^2\geq \lambda_{k+1} \Vert v\Vert_{L^2(\Omega)}^2.
		\end{equation*}
	\end{itemize}
\end{remark}

The following proposition proves the Palais-Smale condition for some type of assymptotically linear reactions.
\begin{proposition}[Palais-Smale]
	Let $\Omega$ be a smooth domain in $\R^N$ and let $f:\R\to \R$ be a continuous function such that
	\begin{itemize}
		\item $f'(\infty):=\lim\limits_{|t|\to \infty} \frac{f(t)}{t}\in (\lambda_k, \lambda_{k+1})$ where $\lambda_k$ and $\lambda_{k+1}$ are two different consecutive eigenvalues of the Laplacian with homogeneous Neumann boundary condition on $\partial \Omega$.
	\end{itemize}
	Then, the energy functional associated to $f$ satisfies the Palais-Smale condition.
\end{proposition}

\begin{proof}
	Since $f$ is assymptotically linear it follows that the energy functional associated to $f$, which we will call $J$, is well defined and its derivative is of the form ``identity minus compact'' (this follows from adding and substracting $u^2$ to the energy functional). This last observation implies that, in order to check the PS condition, it suffices to check that any PS sequence is bounded (also see Variational methods class-notes). \\
	Let $\{u_n\}_{n\in \N}$ be a (PS)-sequence for $J$, keeping the notation of the previous remark let us write $u_n=w_n+v_n$ with $w_n\in Y^\bot$ and $v_n\in Y$. Since  $\{u_n\}_{n\in \N}$ be a (PS)-sequence we have that for $n$ large
	\begin{equation*}
	\int_{\Omega}|\nabla w_n|^2-\int_{\Omega}|\nabla v_n|^2\leq \Vert w_n+v_n\Vert_{H^1(\Omega)}+\int_{\Omega}f(w_n+v_n)(w_n-v_n).
	\end{equation*}
	Given that $f$ is assymptotically linear there exist a continuous function $h$ such that $f(t)=f'(\infty)t+h(t)$ with $h(t)=o(t)$ as $|t|\to \infty$. This consideration implies that
	\begin{equation*}
	\int_{\Omega}\left(|\nabla w_n|^2-f'(\infty)w_n^2\right)-\int_{\Omega}\left(|\nabla v_n|^2-f'(\infty)v_n^2\right)\leq \Vert w_n+v_n\Vert_{H^1(\Omega)}+\int_{\Omega}h(w_n+v_n)(w_n-v_n).
	\end{equation*}
	Decomposing $v_n=c_n+y_n$, where $c_n$ is the component of $v_n$ in $E_1$ and using the previous remark it follows that 
	\begin{align*}
	\Vert w_n\Vert_{H^1(\Omega)}^2\left(1-\frac{f'(\infty)}{\lambda_{k+1}}\right)\frac{\lambda_{k+1}}{\lambda_{k+1}+1}+\Vert y_n\Vert_{H^1(\Omega)}^2\left(\frac{f'(\infty)}{\lambda_k}-1\right)\frac{\lambda_{2}}{\lambda_{2}+1} +\Vert c_n\Vert_{H^1(\Omega)}^2&\leq \Vert w_n+v_n\Vert_{H^1(\Omega)}\\
	&+\int_{\Omega}h(w_n+v_n)(w_n-v_n).
	\end{align*}
	Hence, by orthogonality there exist $c>0$ such that
	\begin{equation*}
	c\Vert w_n+v_n\Vert_{H^1(\Omega)}^2 \leq \Vert w_n+v_n\Vert_{H^1(\Omega)}+\int_{\Omega}h(w_n+v_n)(w_n-v_n).
	\end{equation*}
	The properties of $h$ implies that there exist $a>0$ such that $|h(t)|\leq \frac{c}{2}|t|+a$. Therefore, from the Cauchy-Schwartz inequality (an by orthogonality again) it follows that
	\begin{equation*}
	\int_{\Omega}h(w_n+v_n)(w_n-v_n)\leq \Vert w_n+v_n\Vert_{H^1(\Omega)}\Vert h(w_n+v_n)\Vert_{L^2(\Omega)}\leq \frac{c}{2}\Vert w_n+v_n\Vert_{H^1(\Omega)}^2+a|\Omega|^\frac{1}{2}\Vert w_n+v_n\Vert_{H^1(\Omega)}.
	\end{equation*}
	Finally, combining the last two inequalities we get that our (PS)-sequence is bounded.

\end{proof}

\begin{proposition}[Positive and negative solution]
	Let $f:\R\to \R$ be a continuous function such that
	\begin{itemize}
		\item $f(0)=0$.
		\item $f$ is differentiable at $0$ and $f'(0)<0$.
		\item $f'(\infty):=\lim\limits_{|t|\to \infty} \frac{f(t)}{t}\in (\lambda_k, \lambda_{k+1})$ where $\lambda_k$ and $\lambda_{k+1}$ are two different  eigenvalues of the Laplacian with homogeneous Neumann boundary condition on $\partial \Omega$.
	\end{itemize}
	Then the problem \ref{eq Neumann problem} has at least one negative solution and at least one positive solution.
\end{proposition}
\begin{proof}
	Let us define the continuous function $f^+$ as $f(t)$ for $t\geq0$ and $f'(0)t$ for $t<0$, and let us define $f^-$ as $f(t)$ for $t\leq0$ and $f'(0)t$ for $t<0$. Hence, since $f$ is assymptotically linear it follows that both of these functions are subcritical, implying that the corresponding energy functionals $J^{+}$ and $J^{-}$ are well defined and that their derivatives are of the form ``identity minus compact''. This last observation implies that, in order to check the PS condition, it suffices to check that any PS sequence is bounded (also see Variational methods class-notes).\\
	
	Let us prove that $J^+$ satisfies the PS condition, the proof for $J^-$ is completely analogous. Let $\{u_n\}_{n\in \N}$ be a (PS)-sequence for $J^+$. If we define $u_n^-:=\min\{u_n,0\}$ (the negative part of $u_n$), it follows that, for $n$ large
	\begin{equation*}
	DJ^+(u_n)(u_n^-)=\int_{\Omega}|\nabla u_n^-|^2-f'(0)\int_{\Omega}(u_n^-)^2\leq \Vert u_n^-\Vert_{H^1(\Omega)}.
	\end{equation*}
	Since $f'(0)<0$ the last inequality proves that $\{u_n^-\}_{n\in \N}$ is bounded in $H^1(\Omega)$ by a constant $K>0$. Fixing the notation $u_n^+:=\max\{u_n,0\}$ we can proceed as in the previous proposition and consider the decomposition  $u_n^+=w_n+v_n$ with $w_n\in Y^\bot$ and $v_n\in Y$, where $Y$ is defined as in \textit{Remark} \ref{eigenvalues}. Hence, since our sequence satisfies the PS conditions, for $n$ large we have that
	\begin{equation*}
	DJ^+(u_n)(w_n-v_n)=DJ^+(u_n^+)(w_n-v_n)+DJ^+(u_n^-)(w_n-v_n)\leq \Vert w_n+v_n\Vert_{H^1(\Omega)}.
	\end{equation*}
	By the boundedness of  $\{u_n^-\}_{n\in \N}$ in $H^1(\Omega)$ we can use the last estimation and Cauchy-Schwartz inequality to conclude that 
	\begin{equation*}
	\int_{\Omega}|\nabla w_n|^2-\int_{\Omega}|\nabla v_n|^2\leq \Vert w_n+v_n\Vert_{H^1(\Omega)}+\int_{\Omega}f(w_n+v_n)(w_n-v_n)+K(1+|f'(0)|)\Vert w_n+v_n\Vert_{H^1(\Omega)}
	\end{equation*}
	and hence
	\begin{equation*}
	\int_{\Omega}|\nabla w_n|^2-\int_{\Omega}|\nabla v_n|^2\leq c\Vert w_n+v_n\Vert_{H^1(\Omega)}+\int_{\Omega}f(w_n+v_n)(w_n-v_n)
	\end{equation*}
	In this point we can repeat the same arguments used in the previous proposition to conclude that $\{u_n^+\}_{n\in \N}$ is also bounded in $H^1(\Omega)$.\\
	
	The next part of the proof consists in showing that $0$ is a strong minimum of the functional $J^+$. First of all, notice that if define $F^+(t):=\int_{0}^{t}f(s)ds$ we can apply Taylor's theorem to $F^+$ to get that for some $\delta>0$ and $|t|<\delta$ it holds
	\begin{equation*}
	F^+(t)=F^+(0)+f(0)t+\frac{f'(0)}{2}t^2+o(t^2)=\left(\frac{f'(0)}{2}+\frac{o(t^2)}{t^2}\right)t^2.
	\end{equation*}
	
	Therefore, there exists $0<\delta'<\delta$ such that for $|t|<\delta'$ we have $F^+(t)\leq \frac{c}{2}t^2$ for some $c<0$. On the other hand, our hypotheses about $f$ implies that the function $\frac{f^+(t)}{t}$ is bounded by a positive constant, thus there exists $c'>0$ such that $F^+(t)\leq \frac{c'}{2}t^2$ for every $t$.\\
	
	Keeping in mind these considerations we can estimate $J^-(u)$ for $u$ small in $H^1(\Omega)$ in the following way
	\begin{align*}
	J_\alpha^-(\alpha+u)\geq \frac{1}{2}\left(\int_{\Omega} |\nabla (u+\alpha)|^2-c\int_{\{|u|<\delta'\}}(u+\alpha)^2-c'\int_{\{|u|\geq\delta'\}}(u+\alpha)^2\right)=\frac{1}{2}\left(\int_{\Omega} |\nabla u|^2-c\int_{\Omega}u^2+(c-c')\int_{\{|u|\geq\delta'\}}u^2\right)
	\end{align*}
	Using the critical continuous embedding, Holder inequality with $p=\frac{2^*}{2}$ and $q=\frac{2^*}{2^*-2}$ and Chebyshev inequality it follows that
	\begin{equation*}
	\int_{\{|u|\geq\delta'\}}u^2\leq C_N\Vert u \Vert_{H^1(\Omega)}^2\left(\frac{\Vert u\Vert_{L^2(\Omega)}}{\delta'^2}\right)^\frac{1}{q}
	\end{equation*}
	Thus, there exists a positive constant $C>0$ and $\varepsilon>0$ such that for $\Vert u \Vert_{H^1(\Omega)}\leq\varepsilon$
	\begin{equation*}
	J^+(u)\geq C \Vert u \Vert_{H^1(\Omega)}^2
	\end{equation*}
	Now, to check the last hypothesis of the mountain pass theorem, let us notice that since $f'(\infty)>0$ there exists constants $a>0$ and $b$ such that, for $t>0$, $F^+(t)>at^2+b$. Hence, let us notice that the sequence of constant functions $\{n\}_{n\in\N}$ satisfies
	\begin{equation*}
	J^+(n)=-\int_{\Omega}F(n)\leq |\Omega|(-an^2+b)\to -\infty,\quad n\to \infty.
	\end{equation*}
	Finally, by the mountain pass theorem $J^+$ has a critical point $\omega$ such that $J(\omega)>0$, implying that $\omega$ is not identically 0. By our assumptions we can conclude that $\omega$ is a classical solution of the problem
	\begin{equation*}
	\begin{cases}
	-\Delta u=f^+(u) \quad \text{in} \; \Omega \, ,\\
	\frac{\partial u}{\partial \hat{n}} = 0    \qquad \hspace{2mm} \enspace \text{on}\, \, \: \partial \Omega,
	\end{cases}
	\end{equation*}
	Hence, if $\omega$ is constant then $\omega$ has to be positive, since $f^+$ has not negative zeroes. If $\omega$ is nonconstant we can apply \textit{Proposition} \ref{qualitative} to conclude that $\omega$ is positive since $f^+(t)>0$ for $t<0$. Finally, since $f^+(t)=f(t)$ for $t>0$ the result follows.
\end{proof}

The next proposition generalizes the previous result for any $\alpha \in \R$, we prove the symmetric case that we did not prove in the last case for the sake of completeness.

\begin{proposition}[Generalization of positive and negative solutions]
	Let $\alpha \in \R$ and let $f:\R\to \R$ be a continuous function such that
	\begin{itemize}
		\item $f(\alpha)=0$.
		\item $f$ is differentiable at $\alpha$ and $f'(\alpha)<0$.
		\item $f'(\infty):=\lim\limits_{|t|\to \infty} \frac{f(t)}{t}\in (\lambda_k, \lambda_{k+1})$ where $\lambda_k$ and $\lambda_{k+1}$ are two different eigenvalues of the Laplacian with homogeneous Neumann boundary condition on $\partial \Omega$.
	\end{itemize}
	Then the problem \ref{eq Neumann problem} has a solution $\omega_1$ such that $\min_{\overline{\Omega}} \omega_1 >\alpha$ and a solution $\omega_2$ such that $\max_{\overline{\Omega}} \omega_2 <\alpha$
\end{proposition}
\begin{proof}
	Let us define the continuous function $f_\alpha^+$ as $f(t)$ for $t\geq \alpha$ and $f'(\alpha)(t-\alpha)$ for $t<\alpha$ and $f_\alpha^-$ as $f(t)$ for $t\leq \alpha$ and $f'(\alpha)(t-\alpha)$ for $t>\alpha$. Hence, since $f$ is assymptoticaly linear it follows that both of these functions are subctitical, implying that we can associate a energy functional to each one. Moreover, the derivative of both functionals are of the form ``indetity minus compact''. As in the previous cases, in order to check the PS condition, it suffices to check that any PS sequence is bounded.\\
	\vspace{1mm}
	
	As we mentioned before the statement of the proposition we will prove the result for $J_\alpha^-$ since the prove for the other case is completely analogous.\\
	Let $\{u_n\}_{n\in \N}$ be a (PS)-sequence for $J_\alpha^-$. If we define $(u_n-\alpha)^+:=\max\{u_n-\alpha,0\}$ (the positive part part of $u_n-\alpha$), it follows that, for $n$ large
	\begin{equation*}
	DJ_\alpha^+(u_n)((u_n-\alpha)^+)=\int_{\Omega}|\nabla (u_n-\alpha)^+|^2-f'(\alpha)\int_{\Omega}((u_n-\alpha)^+)^2\leq \Vert (u_n-\alpha)^+\Vert_{H^1(\Omega)}.
	\end{equation*}
	Since $f'(\alpha)<0$ the last inequality proves that $\{(u_n-\alpha)^-\}_{n\in \N}$ is bounded in $H^1(\Omega)$ by a constant $K>0$. Fixing the notation $(u_n-\alpha)^-:=\min\{u_n-\alpha,0\}$ we can proceed as in the previous proposition and consider the decomposition  $(u_n-\alpha)^-=w_n+v_n$ with $w_n\in Y^\bot$ and $v_n\in Y$, where $Y$ is defined as in \textit{Remark} \ref{eigenvalues}. Hence, since our sequence satisfies the PS conditions, for $n$ large we have that
	\begin{align*}
	DJ_a^+(u_n)(w_n-v_n)=&\int_{\Omega}\nabla(u_n-\alpha)^-\cdot \nabla(w_n-v_n)+\nabla(u_n-\alpha)^+\cdot \nabla(w_n-v_n)-\int_{\Omega}f'(\alpha)(u_n-\alpha)^+(w_n-v_n)\\
	+&\int_{\Omega}f((u_n-\alpha)^-+\alpha)(w_n-v_n)\leq \Vert w_n+v_n\Vert_{H^1(\Omega)}.
	\end{align*}
	By the boundedness of  $\{(u_n-\alpha)^+\}_{n\in \N}$ in $H^1(\Omega)$ we can use the last estimation and Cauchy-Schwartz inequality to conclude that 
	\begin{equation*}
	\int_{\Omega}|\nabla w_n|^2-\int_{\Omega}|\nabla v_n|^2\leq \Vert w_n+v_n\Vert_{H^1(\Omega)}+\int_{\Omega}f(w_n+v_n+\alpha)(w_n-v_n)+K(1+|f'(0)|)\Vert w_n+v_n\Vert_{H^1(\Omega)}
	\end{equation*}
	and hence
	\begin{equation*}
	\int_{\Omega}|\nabla w_n|^2-\int_{\Omega}|\nabla v_n|^2\leq c\Vert w_n+v_n\Vert_{H^1(\Omega)}+\int_{\Omega}f(w_n+v_n+\alpha)(w_n-v_n)
	\end{equation*}
	In this point we can repeat the same arguments used in the previous proposition, which are almost identical but for a linear term introduced by the presence of $\alpha$ in the reaction $f$. After repeating the procedure we conclude that $\{(u_n-\alpha)^-\}_{n\in \N}$ is also bounded in $H^1(\Omega)$.\\
	
	The next part of the proof consists in showing that $\alpha$ is an strong minimum of the functional $J_\alpha^-$. First of all, notice that if define $F_\alpha^-(t):=\int_{\alpha}^{t}f(s)ds$ we can apply Taylor's theorem to $F_\alpha^+$ to get that for some $\delta>0$ and $|t|<\delta$ it holds
	\begin{equation*}
	F_\alpha^-(\alpha+t)=F^-(\alpha)+f(\alpha)t+\frac{f'(\alpha)}{2}t^2+o(t^2)=\left(\frac{f'(\alpha)}{2}+\frac{o(t^2)}{t^2}\right)t^2.
	\end{equation*}
	
	Therefore, there exists $0<\delta'<\delta$ such that for $|t|<\delta'$ we have $F_\alpha^+(\alpha+t)\leq \frac{c}{2}t^2$ for some $c<0$. On the other hand, our hypotheses about $f$ implies that the function $\frac{f_\alpha^-(t)}{t-\alpha}$ is bounded by a positive constant, thus there exists $c'>0$ such that $F_\alpha^-(t)\leq \frac{c'}{2}(t-\alpha)^2$ for every $t$.\\
	
	Keeping in mind these considerations we can estimate $J_\alpha^-(\alpha+u)$ for $u$ small in $H^1(\Omega)$ in the following way
	\begin{align*}
	J_\alpha^-(\alpha+u)\geq \frac{1}{2}\left(\int_{\Omega} |\nabla u|^2-c\int_{\{|u|<\delta'\}}u^2-c'\int_{\{|u|\geq\delta'\}}u^2\right)=\frac{1}{2}\left(\int_{\Omega} |\nabla u|^2-c\int_{\Omega}u^2+(c-c')\int_{\{|u|\geq\delta'\}}u^2\right)
	\end{align*}
	Using the critical continuous embedding, Holder inequality with $p=\frac{2^*}{2}$ and $q=\frac{2^*}{2^*-2}$ and Chebyshev inequality it follows that
	\begin{equation*}
	\int_{\{|u|\geq\delta'\}}u^2\leq C_N\Vert u \Vert_{H^1(\Omega)}^2\left(\frac{\Vert u\Vert_{L^2(\Omega)}}{\delta'^2}\right)^\frac{1}{q}
	\end{equation*}
	Thus, there exists a positive constant $C>0$ and $\varepsilon>0$ such that for $\Vert u \Vert_{H^1(\Omega)}\leq\varepsilon$
	\begin{equation*}
	J_\alpha^-(\alpha+u)\geq C \Vert u \Vert_{H^1(\Omega)}^2
	\end{equation*}
	Now, to check the last hypothesis of the mountain pass theorem, let us notice that since $f'(\infty)>0$ there exist constants $a>0$ and $b$ such that, for $t<\alpha$, $F_\alpha^-(t)>at^2+b$. Hence, let us notice that the sequence of constant functions $\{-n\}_{n\in\N}$ satisfies
	\begin{equation*}
	J_\alpha^-(-n)=-\int_{\Omega}F(-n)\leq |\Omega|(-an^2+b)\to -\infty,\quad n\to \infty.
	\end{equation*}
	Finally, by the mountain pass theorem $J_\alpha^-$ has a critical point $\omega_2$ such that $J(\omega_2)>0$, implying that $\omega_2$ is not identically $\alpha$. By our assumptions we can conclude that $\omega_2$ is a classical solution of the problem
	\begin{equation*}
	\begin{cases}
	-\Delta u=f_\alpha^-(u) \quad \text{in} \; \Omega \, ,\\
	\frac{\partial u}{\partial \hat{n}} = 0    \qquad \hspace{2mm} \enspace \text{on}\, \, \: \partial \Omega,
	\end{cases}
	\end{equation*}
	Hence, if $\omega_2$ is constant then $\omega_2$ has to be smaller than $\alpha$, since $f_\alpha^-$ has not zeroes greater than $\alpha$. If $\omega_2$ is nonconstant we can apply \textit{Proposition} \ref{qualitative} to conclude that the maximum of $\omega_2$ is strictly lower than $\alpha$ since $f_\alpha^-(t)<0$ for $t>\alpha$. Finally, since $f_\alpha^-(t)=f(t)$ for $t<\alpha$ the result follows.
\end{proof}

The following result provides an useful way to obtain solutions (hopefully nontrivial) only assuming some local properties in $f$.

\begin{proposition}[Mountain pass theorem between two trivial minima]
	Let $\alpha,\beta \in \R$ with $\alpha<\beta$ and let $f:\R\to \R$ be a continuous function such that
	\begin{itemize}
		\item $f(\alpha)=f(\beta)=0$.
		\item $f$ is differentiable at $\alpha$ and at $\beta$ with  $f'(\alpha)<0$ and $f'(\beta)<0$.
		
	\end{itemize}
	Then the problem \ref{eq Neumann problem} has a solution $\omega_3$ such that $\min_{\overline{\Omega}} \omega_3 >\alpha$ and  $\max_{\overline{\Omega}} \omega_3 <\beta$.
\end{proposition}
\begin{proof}
	Let us define the following auxiliary continuous function $g:\R\to \R$ as follows
	
	\begin{equation*}
	g(t)=
	\begin{cases}
	f'(\alpha)(t-\alpha), \quad t<\alpha.\\
	f(t),\qquad \quad \quad t\in[\alpha,\beta],\\
	f'(\beta)(t-\beta), \quad t>\beta.
	\end{cases}
	\end{equation*}
	
	Notice that if we define $G(t):=\int_{0}^{t}$ there exist constants $a<0$ and $b\in \R$ such that $G(t)\leq \frac{a}{2}t^2+b$ implying that the functional
	\begin{equation*}
	I(u):=\int_{\Omega}\frac{1}{2}|\nabla u|^2-G(u)
	\end{equation*} 
	is well defined, continuous differentiable and coercive in $H^1(\Omega)$, which implies that $I$ satisfies the (PS) condition (again the derivative $I$ is of the form ``identity minus compact'').\\
	
	Without loss of generality let us assume that $I(\alpha)\leq I(\beta)$. Clearly, proceeding as in the previous propositions it can be shown that $\alpha$ and $\beta$ are strong minima of $I$, therefore by the mountain pass theorem we can find a critical point $w_3$ of $I$ such that $I(w_3)>\max\{I(\alpha),I(\beta)\}$. On the other hand, by regularity theory we can conclude that $\omega_3$ is a classical solution of the problem
	
	\begin{equation*}
	\begin{cases}
	-\Delta u=g(u) \quad \text{in} \; \Omega \, ,\\
	\frac{\partial u}{\partial \hat{n}} = 0    \qquad \hspace{2mm} \enspace \text{on}\, \, \: \partial \Omega.
	\end{cases}
	\end{equation*}
	Hence, if $\omega_3$ is constant then $\omega_3$ has to be greater than $\alpha$ and smaller that $\beta$, since $g$ has not in the complement of the interval $[\alpha,\beta]$. If $\omega_3$ is nonconstant we can apply \textit{Proposition} \ref{qualitative} to conclude that the maximum of $\omega_3$ is strictly lower than $\beta$ since $g<0$ for $t>beta$, simiarly the minimum of $\omega_3$ must be strictly greater than $\alpha$. Finally, since $g(t)=f(t)$ for $t\in[\alpha,\beta]$ the result follows.
	
\end{proof}

The following proposition is intended to prove that under some mild assumptions the critical points obtained before by meaas the mountain pass theorem are not trivial. In order to prove this we will require some results due to Helmut Hofer, (see ``A note on the topological degree at a critical point of mountainpass-type'', ``A geometric description of the neighbourhood of a critical point given by the mountain pass theorem'' and ``The topolocial degree at a critical point of the mountain-pass type'').\\

\begin{definition}[Critical point of the mountain-pass type]
	Let $X$ be a real Banach space, let $U\subset X$ a nonvoid open set and let $J\in C^1(U;\R)$. Given $u\in X$ is such that $DJ(u)=0$ and $J(u)=d$ we will say that $u$ is of the mountain-pass type (mp type) if there exist an open neighborhood $W$ of $u$ in $U$ such that for every $V$ open such that $u\in V\subset W$ the open set $\{v\in V| J(v)<d\}$ is nonvoid and non-pathconnected.
\end{definition} 
\begin{remark}
	Connected and path-connectedness in open subsets of a normed linear space are equivalent, since an open set in a normed linear space is locally path-connected,
\end{remark}
First of all let us give conditions to ensure that a trivial solution of \ref{eq Neumann problem} is nondegenerate
\begin{proposition}\label{ nondegenerancy}
	Let $\alpha\in \R$ and let $f:\R\to\R$ be a continuous differentiable function such that
	\begin{itemize}
		\item $f(\alpha)=0$.
		\item $f'(\alpha)\neq \lambda_k$ for $k\in \N$
		\item $f'$ is subcritical.
	\end{itemize}
	Then $\alpha$ is a nondegenerate critical point of the energy functional $J$ associated with $f$.
\end{proposition}
\begin{proof}
	Let us notice that we can write $D^2J(\alpha)(u,v)=\langle Lu,v\rangle$ with $Lu:=u-(1+f'(\alpha))T(u)$, where $T:H^1(\Omega)\to H^1(\Omega)$ is the solution operator of the linear problem 
	\begin{equation}\label{eq Neumann linear problem}
	\begin{cases}
	-\Delta u=f \quad \text{in} \; \Omega \, ,\\
	\frac{\partial u}{\partial \hat{n}} = 0    \qquad  \enspace \text{on}\, \, \: \partial \Omega.
	\end{cases}
	\end{equation}
	Since $L$ has the structure ``Identity minus compact'', Fredholm alternative implies that it suffices to prove that $L$ is injective in order to ensure that $\alpha$ is nondegenerate. In this order of ideas consider $u,v\in H^1(\Omega)$ such that $Lu=Lv$, this implies that
	\begin{equation*}
	\frac{1}{f'(\alpha)+1}(u-v)=T(u-v).
	\end{equation*}
	\textbf{Note:} The result is trivial if $f'(\alpha)=-1$.\\
	
	On the other hand, since the eigenvalues of $T$ have the form (see Santiago Correa's thesis) $\frac{1}{1+\lambda_k}$ for some $k\in \N$, then $u=v$.
\end{proof}
In the following proposition we give some conditions to ensure that a trivial solution of \ref{eq Neumann problem} is not of the mp type.

\begin{proposition}\label{ non montain pass type}
	Let $\alpha\in \R$ and let $f:\R\to\R$ be a continuous differentiable function such that
	\begin{itemize}
		\item $f(\alpha)=0$.
		\item $f'(\alpha)> \lambda_k$ for $k\geq 2$
		\item $f'$ is subcritical.
	\end{itemize}
	Then if either $\alpha$ is an isolated critical point of the energy functional $J$ associated with $f$ or $f'(\alpha)\neq \lambda_l$ for every $l\in \N$, then $\alpha$ is not a critical point of the mp type.
\end{proposition}
\begin{proof}
	For this prove we will require a version of the so-called Morse Lemma due to Hofer (Lemma 3,``A note on the topological degree at a critical point of mountainpass-type''). In order to apply this lemma notice that $0$ is an isolated critical point of the functional $\Phi(u):=J(u+\alpha)-J(\alpha)$ and that $\nabla \Phi$ has the form ``identity minus compact''. On the other hand, we have that $D^2\Phi(0)(u,v)=\langle Lu,v\rangle$, with
	\begin{equation*}
	Lu:=u-T(u)
	\end{equation*}
	as in the previous proposition. A straightforward computation show us that the eigenvalues of $L$ are of the form $\left\{\frac{\lambda_l-f'(\alpha)}{\lambda_l+1}\right\}_{l\in \N}$, this implies that dim$H^{-}\geq2$.\\
	Now, Hofer's result implies that there exists a homeomorphism $D:V\to W$ where $V$ and $W$ are neighbourhoods of $0$ in $H^1(\Omega)$ and a $C^1$-map origin preserving $\beta$ defined in a 0-neighbourhood of $H^0$ into $H^{-}\oplus H^{+}$ such that
	\begin{equation*}
	\Phi(D(x+y+z))=\frac{1}{2}(-\Vert x\Vert+\Vert z\Vert)+\Phi(y+\beta y)
	\end{equation*} 
	for $x+y+z \in H^{-}\circ H^0\oplus H^{+}$ small.
	On the other hand, our hypotheses implies that there exists a negative constant $C<0$ and $\varepsilon>0$ such that for $u\in \bigoplus\limits_{i=1}^{k} E_i$ with $\Vert u \Vert_{H^1(\Omega)}\leq\varepsilon$
	\begin{equation*}
	J(u+\alpha)\leq C \Vert u \Vert_{H^1(\Omega)}^2+J(\alpha).
	\end{equation*}
	Hence, it is clear that for any $\delta<\varepsilon$
	\begin{equation*}
	C_{\delta}:=B_\delta(\alpha)\cap \bigoplus\limits_{i=1}^{k} E_i 
	\end{equation*}
	has nonempty intersection with $ \{v\in H^1(\Omega)| J(v)<J(\alpha)\}$. Let us define, $N_\delta:=B_\delta(0)\cap \bigoplus\limits_{i=1}^{k} E_i$, thus, taking $\delta$ small enough it follows that the set $B_\delta:=\left(\alpha+D(N_\delta)\right)\cap \{v\in H^1(\Omega)| J(v)<J(\alpha)\}$ is pathconnected. Indeed, consider $u=\alpha+D(x+y+z)\in B_\delta$, and the curve $a(t)=\alpha+D(x+y+tz)$. Morse Lemma implies that
	\begin{equation*}
	J(D(x+y+tz)+\alpha)=J(\alpha)+\frac{1}{2}(\Vert x\Vert+t^2\Vert z\Vert)+\Phi(y+\beta y)\leq J(u)<J(\alpha).
	\end{equation*} 
	Finally, since the set $D\left(\left\{(x,y)\in H^{-1}\oplus H^0| \Vert x+y\Vert <\delta \right\}\right)$ is pathconnected the result follows. 
\end{proof}

	\section{Leray-Schauder degree}
	In this section we will provide some qualitative information about the solutions found in the previous section, that will be useful to prove some multiplicity results in this section and in the next one. In order to give a complete characterization of the critical groups of the mp type critical points we should prove a version of the Hess-Kato theorem for Neumann boundary value problems. In the following proof we follow the ideas of \cite{DEFIGUEIREDO82} .
	
	\begin{lemma}\label{lemma kato-hass type}
		Let $h\in C(\overline{\Omega})$. If the weak weighted eigenvalue problem associated with
		\begin{equation*}
		\begin{cases}
		-\Delta u(x)+u(x)=\mu h(x)u(x) \quad \text{in} \; \Omega \, ,\\
		\frac{\partial u}{\partial \hat{n}} = 0    \qquad \hspace{2mm} \enspace \qquad \qquad \qquad \, \, \, \, \text{on}\,  \: \partial \Omega,
		\end{cases}
		\end{equation*}
		has a smallest positive eigenvalue $\mu$, then all the eigenfunctions associated with $\mu$ are positive and $\mu$ is simple.
	\end{lemma}
	\begin{proof}
		Let us assume that $u\in H^1(\Omega)$ is a eigenfunction associated with $\mu$. Then we have that for every $v\in H^1(\Omega)$
		
		\begin{equation}\label{eq weak kato-hess}
		\int_{\Omega}\nabla u\cdot  \nabla v +uv=\mu\int_{\Omega}huv
		\end{equation}
		Consider the linear bounded operator $T:L^2(\Omega)\to H^1(\Omega)$ defined through Riesz representation theorem by the expression
		\begin{equation*}
		\langle T(f),v\rangle=\int_{\Omega}hfv,\qquad \forall f\in L^2(\Omega), \forall v\in H^1(\Omega).
		\end{equation*}
		By the compact  Sobolev embeddings, $T:H^1(\Omega)\to H^1(\Omega)$ is compact and since
		for every $u,v\in  H^1(\Omega)$
		\begin{equation*}
		\langle T(u),v\rangle=\int_{\Omega}huv=\langle T(v),u\rangle,
		\end{equation*}
		$T$ is also self-adjoint. Notice that the biggest positive eigenvalue of $T$ corresponds with (is the reciprocal of) the smallest positive eigenvalue of \ref{eq weak kato-hess} which exists by our hypotheses.This biggest positive eigenvalue of $T$ is characterized by
		
		\begin{equation*}
		\frac{1}{\mu}=\sup[\langle Tx,x\rangle. \Vert x\Vert=1],
		\end{equation*}
		moreover, if $u\in H^1(\Omega)$ with $\Vert u\Vert=1$ attains the supremum above then $u$ is an eigenfunction associated with that eigenvector. Let $u$ be a eigenfunction associated with $\mu_1$, let us suppose by contradiction that $u$ is sign changing, and let $u^+$ and $u^{-}$ its negative and positive parts, respectivelly. Since $u$ is sign changing $u^+$ and $u^-$ are not 0 implying that
		\begin{equation*}
		\frac{1}{\mu}=\langle Tu,u\rangle=\Vert u^+\Vert^2\langle T\frac{u^+}{\Vert u^+\Vert},\frac{u^+}{\Vert u^+\Vert}\rangle+\Vert u^-\Vert^2\langle T\frac{u^-}{\Vert u^-\Vert},\frac{u^-}{\Vert u^-\Vert}\rangle\leq \frac{1}{\mu}
		\end{equation*}
		and the equality only holds if both, the normalized negative and positive parts attains the supremum. Which implies that there exists a positive eigenfunction associated with $\frac{1}{\mu}$. Let $\omega$ be  the normalized positive part of $u$, our previous considerations implies that $\omega$ satisfies the equation
		\begin{equation}
		\int_{\Omega}\nabla \omega\cdot  \nabla v =\int_{\Omega}\left(\mu h-1\right)\omega v
		\end{equation}
		By standard regularity theory \cite{MMP} we conclude that $\omega$ also solves \ref{eq Neumann problem} with $f(x,t)=\left(\frac{1}{\mu_1}h(x)-1\right)t$.\\
		But from the \textit{Proposition} \ref{qualitative} if $\omega$ is not constant it follows that $f(\min_{\overline{\Omega}} \omega,x)<0$ for at least one $x\in \overline{\Omega}$ implying that $\min_{\overline{\Omega}} \omega>0$. In any case this contradicts our assumption that $u$ is sign changing.\\
		
		The previous observations let us conclude that for any pair $u$ and $v$ eigenvalues associated with $\mu$ the sets
		
		\begin{equation*}
		\left\{\alpha \in \R\bigg| u+\alpha v\geq 0 \right\}, \qquad \left\{\alpha \in \R\bigg| u+\alpha v\leq 0 \right\}.
		\end{equation*}
		are not empty and closed. Hence, by connectedness of $\R$ there exists $\alpha \in \R$ such that $u=\alpha v$.
	\end{proof}
	\begin{corollary}\label{corollary crticial groups mp}
		Suppose that $f$ is continous differentiable and that $f'$ is subcritical. If $u_0$ is an isolated critical point of $J$ of the mp type then
		\begin{equation*}
		\text{Rank}\left(C_q(J,u_0)\right)=\delta_{q,1}
		\end{equation*}
	\end{corollary}
	\begin{proof}
		Notice that  if $\lambda$ is a nonpositive eigenvalue of $D^2J(u_0)$ and $u$ is an eigenfuction associated with $\lambda$ then
		\begin{equation*}
		\int_{\Omega}\nabla u\cdot \nabla v -f'(u_0)uv=\lambda\int_{\Omega}\nabla u\cdot \nabla v+uv
		\end{equation*}
		imlying that
		\begin{equation*}
		\int_{\Omega}\nabla u\cdot  \nabla v +uv=\frac{1}{1-\lambda}\int_{\Omega}(f'(u_0)+1)uv
		\end{equation*}
		for every $v\in H^1(\Omega)$. Hence, the smallest negative eigenvalue of our problem corresponds to the smallest positive eigenvalue of the weighted problem
		\begin{equation*}
		\int_{\Omega}\nabla u\cdot  \nabla v +uv=\mu\int_{\Omega}huv
		\end{equation*}
		with $h=f'(u_0)+1 \in C(\overline{\Omega})$. Thereby, if the smallest eigenvalue of $D^2 J(u_0)$ is nonpositive it must be simple.\\
		
		Finally, the previous remark combined with Theorem 1.6 of \cite{CHANGKC93} implies the result. 
	\end{proof}
	
	Since all the solutions that corresponds to critical points of the mp type  obtained so far corresponds to critical points of truncated functionals it is necessary to prove that their critical groups are preserved by the non trucated energy functional.
	
	\begin{lemma}\label{lemma invariance of critical groups}
		Suppose that $f$ and $g$ are continuous differentiable functions that coincides in the closed interval $[a,b]$ such that $f'$ and $g'$ are subcritical. If $u_0$ is critical point of the energy functional $I$ associated with $g$ such that 
		\begin{equation*}
		a<\inf_{x\in\overline{\Omega}}u(x)\leq \sup_{\overline{x\in\Omega}}u(x) <b,
		\end{equation*}
		then $u_0$ is an critical point of $J$. Moreover, $u_0$ is isolated as critical point of $J$ if and only if $u_0$ is isolated as critical point of $I$ and in this case $C_q(u_0, J)=C_q(u_0, I)$ for all $q\in \Z$. 
	\end{lemma}
	\begin{proof}
		By standard regularity theory (see \cite{MMP}) $u_0$ is also a critical point of $J$. Suppose, by contradiction, that there exists a sequence of critical points of $J$ $\{u_n\}_{n\in \N}$ converging to $u_0$ in $H^1(\Omega)$. Corollary 8.6 and Theorems 8.10 of \cite{MMpP} implies that there exists a constant $C>0$ such that $\Vert u_n \Vert_{C^{1,\alpha}(\overline{\Omega})}<C$ for some $\alpha>0$, hence, Arzel\'a-Ascoli theorem implies that, up to a subsequence, $u_n\to u$ in the $C^1$ topology. Our hypotheses implies that for $n$ large $f(u_n)=g(u_n)$ implying that each $u_n$ is a critical point of $I$ which contradicts the fact that $u$ is an isolated critical point of $I$.\\
		
		On one hand notice that $J$ and $I$ coincide in a neighbourhood of $u_0$ in the $C^1$ topology, on the other hand Theorem 5.1.16 (and the subsequent remark) of \cite{CHANGKC06}  implies that $C_q(u_0, J|_{C^1})=C_q(u_0, J)$ and $C_q(u_0, I|_{C^1})=C_q(u_0, I)$ for all $q\in \Z$, which concludes the proof.
		
	\end{proof}
	Now we turn into the computation of the global degree for the case when $f$ is asymptotically linear and non resonant.
	\begin{proposition}\label{prop boundedness solutions}
		Let $\Omega$ be a smooth domain in $\R^N$ and let $f:\R\to \R$ be a continuous differentiable function with $f'$ subcritical such that
		\begin{itemize}
			\item $f'(\infty):=\lim\limits_{|t|\to \infty} \frac{f(t)}{t}\in (\lambda_k, \lambda_{k+1})$ where $\lambda_k$ and $\lambda_{k+1}$ are two different consecutive eigenvalues of the Laplacian with homogeneous Neumann boundary condition on $\partial \Omega$.
		\end{itemize}
		Then,  if we define the linear homotopy $h(\lambda,t):=\lambda f'(\infty)t+(1-\lambda)f(t)$, then there exists $R>0$ such that all the critical values of the energy functional $J_\lambda$ associated to $h(\lambda,\cdot)$  belong to $B_R(0)$ every $\lambda\in [0,1]$. Moreover $d(\nabla J,B_R(0),0)=(-1)^k$.
	\end{proposition}
	\begin{proof}
		Let us suppose, by contradiction, that there exists a sequence $\{u_n\}_{n\in \N}$ of critical points of $J_{\lambda_n}$ with $\lambda_n\in [0,1]$ such that $\Vert u_n\Vert_{H^1(\Omega)}\to \infty$ as $n\to \infty$. Hence we have that for every $v\in H^1(\Omega)$
		\begin{equation*}
		\int_{\Omega} \nabla \left(\frac{u_n}{\Vert u_n\Vert}\right)\cdot \nabla v= \int_{\Omega}  \left(\frac{\lambda_n f'(\infty)u_n+(1-\lambda_n)f(u_n)}{\Vert u_n\Vert}\right) v
		\end{equation*}
		Given our assumptions about $f$ we can rewrite it as
		\begin{equation*}
		f(t)=tf'(\infty)+g(t)
		\end{equation*}
		with $g(t)=o(t)$ as $|t|\to \infty$, getting
		\begin{equation*}
		\int_{\Omega} \nabla \left(\frac{u_n}{\Vert u_n\Vert}\right)\cdot \nabla v= f'(\infty)\int_{\Omega}  \left(\frac{u_n}{\Vert u_n\Vert}\right) v+\int_{\Omega}  \left((1-\lambda_n)\frac{g(u_n)}{\Vert u_n\Vert}\right) v
		\end{equation*}
		Given that $H^1(\Omega)$ is reflexive we have that $\left\{\frac{u_n}{\Vert u_n\Vert}\right\}_{n\in \N}$, converges weakly to some $u\in H^1(\Omega)$  up to a subsequence,. On the other hand, given $\varepsilon>0$ there exists $K>0$ such that if $|t|>k$ then $g(t)/t <\varepsilon$, thus
		\begin{align*}
		\left|\int_{\Omega}  \left(\frac{g(u_n)}{\Vert u_n\Vert}\right) v\right|&\leq\left|\int_{|u_n|\leq K}  \left(\frac{g(u_n)}{\Vert u_n\Vert}\right) v\right|+\left|\int_{|u_n|>K}  \left(\frac{g(u_n)}{u_n}\right)\frac{u_n}{\Vert u_n\Vert} v\right|\\
		&C_\Omega\frac{\sup\{ |g(t)| \big| t\in [-K,K]\}}{\Vert u_n\Vert}\Vert v\Vert+\varepsilon\Vert v\Vert.
		\end{align*}
		Hence, we get
		\begin{equation*}
		\int_{\Omega} \nabla u\cdot \nabla v= f'(\infty)\int_{\Omega} u  v, \qquad \forall v\in H^1(\Omega).
		\end{equation*}
		Let us see that $u\neq 0$. Arguing by contraduction let us suppose that $u=0$, this implies that
		\begin{equation*}
		\Vert u_n\Vert_{H^1(\Omega)}^2=\int_{\Omega} \left(\lambda_n f'(\infty)u_n+(1-\lambda_n)f(u_n)+1\right)u_n
		\end{equation*}
		thereby, by the Cauchy-Schwartz inequality we get 
		\begin{equation*}
		1=\int_{\Omega} \left(\dfrac{f(u_n)+1}{\Vert u_n\Vert}\right)\frac{u_n}{\Vert u_n\Vert}\leq \left(\lambda_n f'(\infty)\left\Vert \dfrac{u_n}{\Vert u_n\Vert} \right\Vert_{L^2(\Omega)} +\left\Vert \dfrac{(1-\lambda_n)g(u_n)+1}{\Vert u_n\Vert} \right\Vert_{L^2(\Omega)}\right)\left\Vert \dfrac{u_n}{\Vert u_n\Vert} \right\Vert_{L^2(\Omega)}.
		\end{equation*}
		Notice that our previous estimates show, in fact, that  $\Vert \dfrac{g(u_n)}{\Vert u_n\Vert} \Vert_{L^2(\Omega)}\to 0$ as $n\to \infty$. On the other hand, since $\dfrac{u_n}{\Vert u_n\Vert}$ converges strongly to 0 in $L^2(\Omega)$, up to a subsequence, we get
		\begin{equation*}
		1\leq \left( f'(\infty)\left\Vert \dfrac{u_n}{\Vert u_n\Vert} \right\Vert_{L^2(\Omega)} +\left\Vert \dfrac{(1-\lambda_n)g(u_n)+1}{\Vert u_n\Vert} \right\Vert_{L^2(\Omega)}\right)\left\Vert \dfrac{u_n}{\Vert u_n\Vert} \right\Vert_{L^2(\Omega)}\to 0,
		\end{equation*}
		as $n\to \infty$.\\
		
		The fact that $u\neq 0$ leads us to the desired contradition given that $f'(\infty)$ is not an eigenvalue of our problem.\\
		
		Given the invariance under homotopy of the degree (Theorem 3.3.1 of \cite{KESAVAN}) it suffices to compute the degree of $\nabla J_1$ in $B_R(0)$ with respect to $0$. Notice that the only critical point of $J_1$ is $0$, hence it suffices to find the number of negative eigenvalues of $0$. Since
		\begin{equation*}
		D^2J_1(0)(u,v)=\langle\mu u, v\rangle \implies (1-\mu)\langle u, v\rangle=(f'(\infty)+1)\int_{\Omega} uv
		\end{equation*}
		It follows that the negative eigenvalues of $D^2J_1(0)$ satisfies 
		\begin{equation*}
		\frac{f'(\infty)+1}{1-\mu}=1+\lambda_i \implies \mu=\frac{\lambda_i-f'(\infty)}{1+\lambda_i}.
		\end{equation*}
		Hence, by our assumptions ($f'(\infty)>\lambda_k$) it follows that $d(\nabla J,B_R(0),0)=(-1)^k$.
	\end{proof}
	
	\begin{proof}[Proof of the first multiplicity result]
		Our hypotheses implies that there exist 5 trivial solutions $m_1,m_2,a_1,a_2$ and $a_3$ . Theorem truncated interval implies the existence of two nontrivial solutions $u_1$ and $u_2$ with range contained in $(-\infty, m_1)$ and $( m_2, \infty)$, respectivelly. Analogously, Theorem truncated interval2 gurantees the existence of a nontrivial solution $u_3$ which range is contained in $(m_1, m_2)$.\\
		
		Since the critical points of $J$ are isolated we can find for each critical point $c$ a ball $B_c$ containing it and no other critical point.  \textit{Corollary} \ref{corollary crticial groups mp} and \textit{Lemma} \ref{lemma invariance of critical groups} implies that $d(\nabla J,B_{u_i},0)=-1$ for $i=1,2,3$. On the other hand, our assumptions implies that $d(\nabla J,B_{m_i},0)=1$ for $i=1,2$ and that $d(\nabla J,B_{a_i},0)=(-1)^{k_i}$ for $i=1,2,3$, with at least one $k_i=k$. Arguing by contradiction, let us suppose that there are not more solutions. The excition property of the Leray-Schauder degree implies that
		\begin{equation*}
		(-1)^k=d(\nabla J,B_{R},0)=\sum_{i=1}^{3}d(\nabla J,B_{u_i},0)+\sum_{i=1}^{2}d(\nabla J,B_{m_i},0)+\sum_{i=1}^{3}d(\nabla J,B_{a_i},0)=-1+\sum_{i=1}^{3}(-1)^{k_i}.
		\end{equation*} 
		Thus, it exists integers $a$ and $b$ such that
		\begin{equation*}
		1=(-1)^a+(-1)^b
		\end{equation*} 
		which is clearly a contradiction. Finally, there exists at least another nontrival critical point of $J$.
	\end{proof}
	
	\section{Lyapunov-Schmidt reduction method}
	\subsection{Important things to be mentioned somewhere before this section.}
		\begin{enumerate}
			\item Let $\Omega \subset \R$ be a bounded domain with smooth boundary.
			\item A variational formulation of the problem (*) leads to an energy functional of the form:
				\begin{equation}\label{Functional}
					J\left(u\right)=\int_{\Omega}{\left(|\nabla u|^2 - F\left(u\right)\right)}\,dx,
				\end{equation}				 
				where $F\left(t\right) \eqdef \int_{0}^{t}{f\left(s\right)}\,ds$.
			\item For a proof of the next result, see \cite{CH}.
			\begin{lemma} \label{DegCrit}
					Let $H$ be a real Hilbert space, and let $J \in C^2\left(H,\R\right)$ be a function satisfying the (PS) condition. Assume that $\nabla J\left(x\right)=I-T$, where $T$ is a compact mapping, and $u_0$ is an isolated critical point of $J$. Then we have
						\begin{equation*}
							d\left(\nabla J,u_0\right)=\sum_{q=0}^{\infty}{\left(-1\right)^q\dim{C_q\left(J,u_0\right)}}.
						\end{equation*}					
                   \end{lemma}
		\end{enumerate}	
	\subsection{Preliminaries and notation.}
			The set of eigenvalues of $-\Delta$ with homogeneous Neumann boundary condition on $\partial \Omega$ can be written as an increasing non-negative sequence $\left\{\lambda_i\right\}_{i=1}^{\infty}$ such that $\mathop {\lim }\limits_{i \to \infty } {\lambda _i} = \infty$. In addition, if $\left\{\varphi_i\right\}_{i=1}^{\infty}$ denotes the corresponding sequence of eigenfunctions and $E_i$ stands for the eigenspace generated by $\varphi_i$,  the following are well-known facts (see, for example, Motreanu-Papageorgiou \cite{MMP}):
			\begin{enumerate}
				\item $\lambda_1=0$ is simple and $E_1$ is formed exclusively by constant functions.
				\item $\left\{\varphi_i\right\}_{i=1}^{\infty}$ is an orthogonal basis of $L^2\left(\Omega\right)$ and $H^1\left(\Omega\right)$, respectively.
				\item Let $X$ denote the subspace spanned by $\left\{\varphi_1,\varphi_2,\cdots,\varphi_k\right\}$ and $Y$ its orthogonal complement in $H^1\left(\Omega\right)$. As a consequence of the variational characterization of the eigenvalues we have the two Poincar\'e-like inequalities:
					\begin{equation}\label{PoincareX}
						\Vert \nabla v\Vert_{L^2(\Omega)}^2\leq \lambda_k \Vert v\Vert_{L^2(\Omega)}^2 \quad \forall v\in X,
					\end{equation}
				and 
					\begin{equation}\label{PoincareY}
						\Vert \nabla v\Vert_{L^2(\Omega)}^2\geq \lambda_{k+1} \Vert v\Vert_{L^2(\Omega)}^2 \quad \forall v \in Y.
					\end{equation}
		\end{enumerate}	 
			 	
		Now, we introduce the so-called Lyapunov-Schmidt reduction method.  Under appropriate conditions this technique provides a general procedure to transform a variational infinite-dimensional problem into an equivalent (often easy-to-solve) finite-dimensional  problem. For further discussion, see Castro \cite{C}, Castro-Lazer \cite{CL} and the references therein.
		
		\begin{theorem}\label{LS-T}
			Let $H$ be a real Hilbert space. Let $X$ and $Y$ be closed subspaces of $H$  such that $H=X\oplus Y$. Assume that $J: H \to \R$ is a functional of class $C^1$. If there is a constant $m>0$ such that
			\begin{equation}\label{RedHyp}
				\left\langle \nabla J\left(x+y_1\right)-\nabla J\left(x +y_2\right) ,y_1-y_2\right\rangle \geq m \Vert y_1-y_2\Vert^2 \quad \mathrm{for \; all} \quad x \in X, \; y_1,y_2 \in Y,
			\end{equation}
			then: 
			\begin{enumerate}[leftmargin=*,label=(\roman*)]
				\item \label{RedCont} there exists a continuous function $\psi :X\to Y$ such that 
					\begin{equation*}
						J\left(x+\psi\left(x\right)\right)=\min_{y\in Y}{J\left(x+y\right)}.							
					\end{equation*}	
				\item The function 
					\begin{alignat*}{2}
						\widetilde{J}:\, & X  && \longrightarrow \R\\
			                 			 & x  && \longrightarrow \widetilde{J}\left(x\right) \eqdef J\left(x+\psi\left(x\right)\right)
					\end{alignat*}								
					is of class $C^1$, and
					\begin{equation*} 
						\left\langle \nabla \widetilde{J}\left(x_1\right),x_2\right\rangle=\left\langle\nabla J\left(x_1+\psi\left(x_1\right)\right),x_2\right\rangle \quad \mathrm{for \; all} \quad x_1,x_2 \in X.
					\end{equation*}
					Moreover, $\psi\left(x\right) \in Y$ is the unique element satisfying 
					\begin{equation*}
						\left\langle\nabla J\left(x+\psi\left(x\right)\right),y\right\rangle=0 \quad \mathrm{for \; all} \quad y \in Y.
					\end{equation*}
				\item \label{CritPoint} An element $x_0\in X$ is a critical point of $\widetilde{J}$ if and only if $u_0\eqdef x_0+\psi\left(x_0\right)$ is a critical point of $J$.
			\end{enumerate}
		\end{theorem}
		 Throughout this section, $X$ will denote the vector space spanned by $\left\{\varphi_1,\varphi_2,\cdots,\varphi_k\right\}$ and $Y$ its orthogonal complement in $H^1\left(\Omega\right)$. Similarly, $J:H^1\left(\Omega\right)\to \R$, $\widetilde{J}:X\to \R$ and $\psi:X\to Y$ will denote the functions given by Theorem \ref{LS-T}.\\
		 		 
        The next proposition provides sufficient conditions to ensure the conclusions of the previous theorem.
		\begin{proposition}\label{RedHypoFun}
			Let $f:\R\to \R$ be a differentiable function. If there exists a constant $\gamma$ such that $f'\left(t\right) \leq \gamma < \lambda_{k+1}$ for some eigenvalue $\lambda_{k+1}$ of $-\Delta$ with Neumann boundary condition, then $J$ verifies the hypothesis of the Theorem \ref{LS-T}.
		\end{proposition}		
		\begin{proof}
			Fix elements $x\in X$ and $y_1,y_2 \in Y$. In order to prove the desired result, we proceed as follows:
			\begin{align*}
				\left\langle \nabla J\left(x+y_1\right)-\nabla J\left(x +y_2\right) ,y_1-y_2\right\rangle_{H^1(\Omega)}^2 &=\int_{\Omega}{\left(|\nabla(y_1-y_2)|^2-(f(x+y_1)-f(x+y_2))(y_1-y_2)\right)}\, dx\\
		&\geq \int_{\Omega}{\left(|\nabla(y_1-y_2)|^2-\gamma(y_1-y_2)^2\right)} \, dx \\ 
		&\geq \left(1-\frac{\gamma}{\lambda_{k+1}}\right)\frac{\lambda_{k+1}}{1+\lambda_{k+1}}\Vert y_1-y_2\Vert_{H^1(\Omega)}^2.
			\end{align*}
			Then \eqref{RedHyp} holds with $m=\left(\lambda_{k+1}-\gamma\right)/\left(1+\lambda_{k+1}\right)$.
		\end{proof}
		\subsection{Properties.}
		The following result describes a local property of $\widetilde{J}$ in connection with trivial solutions of the problem.
		\begin{proposition} \label{MaxMinTriv}
			Under the assumptions of Proposition \ref{RedHypoFun}, if we additionally suppose that $\alpha \in \R$ is such that $f\left(\alpha\right)=0$ and $f'\left(\alpha\right)\in \left(\lambda_{\ell},\lambda_{\ell+1}\right)$ with $\ell<k$, then $\alpha$ is a strict local maximizer of $\widetilde{J}$ on $\bigoplus_{i=1}^{\ell} E_i$ and $\alpha$ is a strict local minimizer of $\widetilde{J}$ on $\bigoplus_{\ell+1}^{k} E_i$.
		\end{proposition}
		\begin{proof}
			Applying Taylor's theorem to $F\left(t\right)$, we can find positive constants $\varepsilon_1$ and $\beta>\lambda_{\ell}$ such that 
			\begin{equation*}
				F\left(\alpha+t\right)=F\left(\alpha\right)+\left(\frac{f'\left(\alpha\right)}{2}+\frac{o\left(t^2\right)}{t^2}\right)t^2 \geq F\left(\alpha\right)+\frac{\beta}{2}t^2
			\end{equation*}
			for all $0<|t|<\varepsilon_1$. Choose any $u \in \bigoplus_{i=1}^{\ell} E_i$ with $\Vert{u}\Vert_{H^1\left(\Omega\right)}<\varepsilon_1$. If we define $c'= \frac{1}{2}\left(\lambda_{\ell}-\beta\right)/\left(\lambda_{\ell}+1\right)<0$, then
			\begin{equation*}
				J\left(u+\alpha\right)\leq J\left(\alpha\right)+\frac{1}{2}\int_{\Omega}{|\nabla u|^2}\,dx-\frac{\beta}{2}\int_{\Omega}{u^2}\,dx \leq J\left(\alpha\right)+\frac{1}{2}\left(\lambda_{\ell}-\beta\right)\Vert{u}\Vert_{L^2\left(\Omega\right)}^{2}=J\left(\alpha\right)+c'\Vert{u}\Vert^2_{H^1\left(\Omega\right)}.
			\end{equation*}
			Hence, from Theorem \ref{LS-T}\ref{RedCont} and the previous inequality 
			\begin{equation*}
				\widetilde{J}\left(\alpha+u\right)\leq J\left(\alpha+u\right)\leq \widetilde{J}\left(\alpha\right)+c'\Vert{u}\Vert^2_{H^1\left(\Omega\right)}
			\end{equation*}
			since $\widetilde{J}\left(\alpha\right)=J\left(\alpha\right)$. This shows $\alpha$ is a strict local maximizer of $\widetilde{J}$ on $\bigoplus_{i=1}^{\ell} E_i$. Proceeding similarly, it can be shown that for some $\varepsilon_2>0$ and every $u\in \bigoplus_{\ell+1}^{k} E_i$ satisfying $\Vert{u}\Vert_{H^1\left(\Omega\right)}<\varepsilon_2$ we may find a constant $c>0$ such that
			\begin{equation*}
				J\left(u+\alpha\right)\geq J\left(\alpha\right)+c\Vert{u}\Vert^2_{H^1\left(\Omega\right)}.
			\end{equation*}		
		On the other hand, from the continuity of $\psi: X \to Y$ and the fact that $\psi\left(\alpha\right)=0$, we may write $\Vert{u+\psi\left(\alpha+u\right)}\Vert<\varepsilon_2$ for some $2\eta<\varepsilon_2$ and any $u \in \bigoplus_{i=\ell+1}^{k} E_i$ with $\Vert{u}\Vert_{H^1\left(\Omega\right)}<\eta$. Consequently, 
		\begin{align*}
			\widetilde{J}\left(\alpha+u\right)=J\left(\alpha+u+\psi\left(\alpha+u\right)\right)&\geq J\left(\alpha\right)+c\Vert{u+\psi\left(\alpha+u\right)}\Vert_{H^1\left(\Omega\right)}\\&\geq \widetilde{J}\left(\alpha\right)+c\Vert{u}\Vert^2_{H^1\left(\Omega\right)}+c\Vert{\psi\left(\alpha+u\right)}\Vert^2_{H^1\left(\Omega\right)}\\& \geq \widetilde{J}\left(\alpha\right)+c\Vert{u}\Vert^{2}_{H^1\left(\Omega\right)}
		\end{align*}		
		which shows the result.
		\end{proof}		
		
	\begin{definition}\label{mp}
		Let $X$ be a real Banach space and $U \subset X$ a nonempty open set. If  $J\in C^1(U;\R)$ and $u_0$ is a critical point of $J$, then $u_0$ is called of mountain pass type (mp type from now on) if there exists an open neighborhood $W \subset U$ of $u_0$ such that for every open neighborhood $V \subset W$ of $u_0$, $J^{-1}\left(-\infty,J\left(u_0\right)\right)\cap V \neq \varnothing$ and $J^{-1}\left(-\infty,J\left(u_0\right)\right)\cap V$ is not path-connected.
	\end{definition}

	The next theorem explores essential properties regarding the Lyapunov-Schmidt reduction method and the Leray-Schauder degree.
	\begin{theorem}\label{RedCrit}
		Let $f:\R\to \R$ be a continuously differentiable function. If there exists a constant $\gamma$ such that $f'\left(t\right) \leq \gamma < \lambda_{k+1}$ and $f'\left(\infty\right)\in \left(\lambda_{k},\lambda_{k+1}\right)$,  then there exists at least one solution of the problem (*), say $u_{10}$, and, if isolated, $C_q(J,u_{10})=\delta_{q,k}\R$.
	\end{theorem}
	\begin{proof}
		Since $f'\left(\infty\right) \in \left(\lambda_{k},\lambda_{k+1}\right)$, we can safely assume that there exists constants  $\zeta$ and $\beta>\lambda_{k}$ such that 
		\begin{equation*}
			F\left(t\right)\geq \frac{\beta}{2}t^2+\zeta \quad \forall t \in \R.
		\end{equation*}
		Hence, 
		\begin{equation*}
			J\left(u\right)=\frac{1}{2}\int_{\Omega}{|\nabla u|^2}-\int_{\Omega}{F\left(u\right)}\,dx \leq \frac{1}{2}\Vert \nabla u\Vert_{L^2(\Omega)}^2- \frac{\beta}{2}\int_{\Omega}{u^2}\,dx - \zeta |\Omega|.
		\end{equation*}	
		At this point, it is worth recalling that $\Vert \nabla x\Vert_{L^2(\Omega)}^2\leq \lambda_k \Vert x\Vert_{L^2(\Omega)}^2$ for every $x\in X$ and therefore
		\begin{equation*}
			J\left(x\right)\leq \frac{1}{2}\left(1-\frac{\beta}{\lambda_k}\right)\frac{\lambda_{k+1}}{1+\lambda_{k+1}}\Vert{x}\Vert^2_{H^1\left(\Omega\right)}-\zeta|\Omega|\to -\infty \quad \mathrm{as} \quad \Vert x\Vert_{H^1\left(\Omega\right)} \to \infty.
		\end{equation*}			
		 Moreover, since $\widetilde{J}\left(x\right)\leq J\left(x\right)$ we get $\widetilde{J}\left(x\right)\to -\infty$ as $\Vert x\Vert_{H^1\left(\Omega\right)} \to \infty$. Thus, by combining the latter inequality with the condition $\dim X < \infty$,  we deduce the existence of some element $x_{10} \in X$ satisfying
		\begin{equation}\label{GlobMax}
			\widetilde{J}\left(x_0\right)=\max_{x\in X}{\widetilde{J}\left(x\right)}=\max_{x\in X}{J\left(x+\psi\left(x\right)\right)}.
		\end{equation}
		Taking $u_{10}=x_{10}+\psi\left(x_{10}\right)$, we see that $u_{10}$ is a critical point of $J$. It remains to check that $C_q(J,u_{10})=\delta_{q,k}\R$. In fact, as illustrated by \eqref{GlobMax}, $x_{10} \in X$ is a global maximum for $\widetilde{J}$ and consequently $C_q(\widetilde{J},x_{10})=\delta_{q,k}\R$. Finally, in light of the invariance of critical groups under the Lyapunov-Schmidt reduction method (See Liu \cite{L}, Lemma 2.3), we conclude that $C_q(J,u_{10})=\delta_{q,k}\R$. 
	\end{proof}		
	
	We now intend to prove another multiplicity result. To this end, it is necessary to distinguish the solution coming from reduction arguments from the solutions found previously. The following proposition addresses this point.
		
	\begin{proposition} \label{dist} Let $\left\{u_i:1\leq i\leq 10\right\}$ the set of solutions of (*) obtained  so far. Denote by $u_1,u_3,u_5$ the trivial solutions with positive slope and by $u_6,u_7,u_8$ the critical points of mp type. Then, under the assumptions of Theorem \ref{RedCrit}, we have:
		\begin{enumerate}[leftmargin=*,label=(\roman*)]
			\item $u_{10}\neq u_i$ for every solution $u_i$ of the mp type.
			\item \label{distTriv} Let $\alpha \in \left\{u_1,u_3,u_5\right\}$ be such that $J(u_i)\leq J\left(\alpha\right)$ for $i=1,3,5$.  If, in addition, we suppose that $f'\left(\alpha\right)\in \left(\lambda_{\ell},\lambda_{\ell+1}\right)$ with $\ell<k$, then $u_{10} \neq u_i$ for $i=1,3,5$.
		\end{enumerate}			
	\end{proposition}

	\begin{proof}
		\mbox{}\par
		\begin{enumerate}[leftmargin=*,label=(\roman*)]
			\item As mentioned above, $C_k(J,u_{10})$ is the only critical group not vanishing identically. In contrast,  for critical points of the mp type we have already proved that $C_q\left(J,u_i\right)=\delta_{q,1}\R$ (See Corollary 2 - Daniel). This makes the  desired distinction.
			\item From Proposition \ref{MaxMinTriv} we know that $\alpha$ is an strict local minimizer of $\widetilde{J}$ on $\bigoplus_{\ell+1}^{k} E_i$, that is, $\widetilde{J}\left(\alpha\right)<\widetilde{J}\left(\alpha+u\right)$ for any $u \in \bigoplus_{i=\ell+1}^{k} E_i$ with $\Vert u\Vert_{H^1\left(\Omega\right)}$ small enough. Thus, 
				\begin{equation*}
					J\left(u_i\right)\leq J\left(\alpha\right)=\widetilde{J}\left(\alpha\right)<\widetilde{J}\left(\alpha+u\right)\leq\max_{x\in X}{\widetilde{J}\left(x\right)}= J\left(x_{10}\right)\leq J\left(u_{10}\right).
				\end{equation*}				
				It follows that $u_{10} \neq u_i$ for every $i\in\left\{1,2,3\right\}$.
		\end{enumerate}				
	\end{proof}
	\begin{remark} A careful look at the inequality $J(u_i)\leq J\left(\alpha\right)$ in Proposition \ref{dist} \ref{distTriv} shows that   

		\begin{equation*}
			J\left(u_i\right)=-F\left(u_i\right)|\Omega|\leq -F\left(\alpha\right)|\Omega|=J\left(\alpha\right),
		\end{equation*}
		which depends entirely on the value of $F\left(t\right)=\int_{0}^{t}{f\left(t\right)}\,dt$, the area under the curve of the nonlinearity $f$.
		 
	\end{remark}
	   In order to establish our last result, observe that $d_{\mathrm{loc}}(\nabla J,u_{10})=\left(-1\right)^k$ according to Lemma \ref{DegCrit}. Having said that, suppose by contradiction that there are no more solutions (keeping the notation as in the previous degree counting). By the excision property of the Leray-Schauder degree we have 
	\begin{equation*}
		\begin{aligned}
			\left(-1\right)^{k}=d\left(\nabla J,B_R,0\right)&=\sum_{i=6}^{8}d(\nabla J,B_{u_i},0)+\sum_{i=1}^{2}d(\nabla J,B_{u_{2i}},0)+\sum_{i=1}^{3}d(\nabla J,B_{u_{2i-1}},0)+d\left(\nabla J,B_{u_{10}},0\right)\\&=-1+(-1)^{k_1}+(-1)^{k_3}+(-1)^{k_5}+\left(-1\right)^{k}.
		\end{aligned}
	\end{equation*}
	which leads to $1=(-1)^{k_1}+(-1)^{k_3}+(-1)^{k_5}$. In other words, if an extra condition is imposed, namely $1\neq \sum_{i=1}^{3}\left(-1\right)^{k_{2i-1}}$, then there exists at least one additional nontrivial solution of (*).


\bigskip
{\bf Acknowledgments: }The research of the first author was supported by the Grants 13-00863S and 18-03253S of the Grant Agency of the Czech Republic and also by the Project LO1506 of the Ministry of Education, Youth and Sports of the Czech Republic. The research of the remaining authros was partially supported by Colciencias, Fondo Nacional de Financiamiento
para la Ciencia, la Tecnolog ??a y la Innovaci ?on Francisco Jos ?e de Caldas. Project ?Ecuaciones
diferenciales dispersivas y el ??pticas no lineales?, Code 111865842951.
 

\end{document}